\documentclass[a4paper,11pt]{article}
\usepackage{multirow}
\usepackage{color}
\usepackage{latexsym}
\usepackage{amssymb}
\usepackage{amsfonts}
\usepackage{amsmath}
\usepackage{amsthm}
\usepackage{indentfirst}
\usepackage{mathrsfs}
\usepackage{verbatim}

\usepackage[textsize=footnotesize]{todonotes}
\usepackage[normalem]{ulem}
\usepackage{mathtools}

\newtheorem {lemma}{Lemma}[section]
\newtheorem {theorem}[lemma]{Theorem}

\newtheorem {proposition}[lemma]{Proposition}
\newtheorem {corollary}[lemma]{Corollary}

\theoremstyle{remark}

\usepackage[latin1]{inputenc}

\usepackage{fullpage}

\author{Luca Ferrari\thanks{Dipartimento di Matematica e Informatica ``U. Dini", University of Firenze, Firenze, Italy\quad \texttt{luca.ferrari@unifi.it}. Partially supported by INdAM -- GNCS project CUP\_E53C23001670001.}\and Francesco Verciani\thanks{Institut f\"ur Mathematik, Universit\"at Kassel, 34132 Kassel, Germany, \, \texttt{francesco.verciani@mathematik.uni-kassel.de}. Partially supported by German Science Foundation grant 522790373.}}

\title{A new generalization of parking functions: $\rho$-Naples parking functions}

\date{}

\begin{document}
	
\maketitle
	
\begin{abstract}
	
	We introduce a new model of parking functions, called \emph{$\rho$-Naples parking functions}, in which each car has a certain number of spots to check preceding its preferred one before looking for available spots by moving forward. This notion, where each car has \emph{its own} number of previous spots to check, generalizes that of Naples parking function, where this number is the same for all cars. We start by showing that, for a given sequence of preferences for each car, a successful parking strategy (that is an assignment of how many backward spots each car is allowed to check) remains successful if we increase the number of backwards spots to check for some of the cars. This means that the set of successful strategies for a given parking preference is an up-set in the lattice of parking strategies, which raises the intriguing question of characterizing such up-sets. Though we are not able to completely solve this problem, we provide a characterization of those successful strategies which minimize the total number of allowed backward steps for a given parking preference, and we determine a lower bound for the number of cars which are allowed to move backwards (also showing that it can be attained). To conclude, we analyze the specific situation in which each car can take at most one backward step, and we completely characterize the set of successful parking preferences in this case.
  
\end{abstract}
	
\section{Introduction}

%
%

Parking functions were introduced by Konheim and Weiss \cite{KW} in a computer science problem, but quickly proved to be extremely interesting in themselves, due to their remarkable combinatorial properties and applications. For instance, parking functions have appeared in contexts like hyperplane arrangements \cite{St}, representation theory \cite{KT} and combinatorial geometry \cite{AW}. The enumerative combinatorics of parking functions is particularly elegant. Classical results include invariance of a parking function under rearrangement and a neat closed formula for their enumeration (more on this below, see also the classical survey \cite{Y}).

Although it is possible to give a purely formal definition of parking functions, a classical model consists of interpreting them as sequences of positive integers describing the preferences of an ordered list of cars that successfully park in an ordered set of parking spots arranged on a one-way street. The classical parking rule states that every car has its preferred parking spot and, if this is already occupied, it continues forward and park in the first available spot (if there is one). In recent years, this parking model has been widely generalized giving rise to a host of new problems, each of which raises intriguing combinatorial questions, especially from the point of view of enumeration. A few dozen articles (at least) have appeared in the last 5 years dealing with the combinatorics of several kinds of generalized parking problems. A nice and entertaining introduction to this sort of problems, accessible also to undergraduates, is the paper \cite{CCHJR}.

Among all proposed generalizations, one of the most studied ones is certainly that of Naples parking functions \cite{B,CHJLRRR}. In this specific parking model, cars are required to check a limited number of parking spots preceding their preferred one (in case this is already occupied) before proceeding forward. We emphasize the fact that the number of backward spots to be checked is fixed, and is the same for all cars. The combinatorics of Naples parking functions turn out to be significantly more challenging than that of classical parking functions. For instance, it is no longer true that, in general, any rearrangement of a Naples parking function is still a Naples parking function. Moreover, in sharp contrast with the classical case, no satisfactory closed formula is known for the enumeration of Naples parking function (while some recursive formulas have been obtained, see for instance \cite{CHJLRRR}).

The goal of the present paper is to extend the Naples parking rule, by setting a specific number of backward spots to check for each car (not necessarily the same). This new parking model generalizes both the classical and the Naples parking models, and poses several intriguing combinatorial questions, which we start investigating here. In particular, we address some issues that are somewhat dual to the typical problems considered for other parking models. Indeed, instead of fixing the parking rule and studying the set of successful parking preferences, we rather fix a parking preference and investigate the space of our parking rules that make that preference successful. It turns out that this set of rules has a nice structure, that we are able to describe only in part.

We remark that the idea of considering this sort of parking functions has been suggested by \cite{CCHJR}, and specifically by Chapter 1.5b. A specific case is then also proposed in Chapter 1.13, though without providing a general framework for the investigation we carry out here. 

After two brief sections on preliminaries (Section \ref{preliminaries}) and main definitions (Section \ref{main}), we introduce the (finite) lattice of parking strategies in Section \ref{lattice} and prove the crucial fact that the set of successful parking strategies for a given preference is an up-set in such a lattice. The rest of the paper is essentially devoted to understand these up-sets, by investigating their minimal elements. The problem of fully describing such minimal elements appears to be extremely challenging, and we are able to provide only partial results. In particular, in Section \ref{min_back_steps} we completely characterize the minimal elements which minimize the number of backward steps, while in Section \ref{min_cars_back} we do the same for minimal elements which minimize the number of cars driving backwards. Section \ref{zeroone} is devoted to the complete analysis of the special case in which cars can take at most one backward step. Finally, Section \ref{further} contains some hints for further investigations.

The main tools we have used to obtain our results are some notions we have introduced in our previous paper \cite{FV}, in particular the \emph{excess function} and the \emph{excess set} of a parking preference. These are defined in Section \ref{preliminaries}, and identify those positions $j$ that are \emph{critical} for a given parking preference, meaning that there are too many cars having preference greater than or equal to $j$, thus preventing the preference from being a classical parking function (but not a Naples parking function). Section \ref{preliminaries} also contains the main properties of the excess function and the excess set, whose complete proofs can be found in \cite{FV}.

\section{Preliminaries}\label{preliminaries}

A \emph{parking preference} $\alpha =(a_1 ,a_2 ,\ldots ,a_n)$ of length $n$ is a $n$-tuple of positive integers between 1 and $n$ to be interpreted as follows: for an ordered list $(c_1 ,c_2 ,\ldots ,c_n )$ of $n$ cars attempting to park on $n$ numbered spots in a street, $a_i$ denotes the preferred spot for car $c_i$. The set of all parking preferences of length $n$ will be denoted $PP_n$. The total number of cars having preference $i$ will be denoted with $|\alpha |_i$. The \emph{classical} parking rule requires that each car tries to park into its preferred spot and, if this is already occupied, then proceeds forward and parks into the first available spot, if any, otherwise it simply exits the street (thus failing to park). If all cars successfully park using the above rule, then we say that $\alpha$ is a \emph{parking function} of length $n$, and we denote the set of all those with $PF_n$. An elegant characterization of parking functions states that $\alpha$ is a parking function if and only if $\sum_{i=1}^{j}|\alpha |_i \geq j$, for all $j\leq n$. The latter condition is equivalent to the fact that the nondecreasing rearrangement $(b_1 ,b_2 ,\ldots ,b_n)$ of $\alpha$ is such that $b_j \leq j$, for all $j\leq n$. As a consequence, we get that any rearrangement of a parking function is a parking function as well.

The classical parking rule can be modified into the \emph{$k$-Naples parking rule}, by allowing each car to check up to $k$ spots before its preferred one, and then proceed forward if it does not find an available place. A \emph{$k$-Naples parking function} is a parking preference such that all cars manage to park following the $k$-Naples parking rule, and the set of $k$-Naples parking functions of length~$n$ will be denoted $PF_{n,k}$. We remark that, unlike classical parking functions, a rearrangement of a Naples parking function is not a Naples parking function in general: for example, $(3,3,2)\in PF_{3,1}$, but $(2,3,3)\notin PF_{3,1}$. 

\bigskip

Next we recall some notions and results described in \cite{FV}. They are all crucial for the rest of the paper.

Given $\alpha \in PP_n$, we denote with $u_{\alpha}(j)\in \mathbb{Z}$ the excess of cars having preference at least $j$ with respect to the number of available spots from $j$ onwards (for $1\leq j\leq n$). More formally,
\[
	u_{\alpha}(j)=\sum_{i=j}^{n}|\alpha|_i -(n-j+1).
\]

The map $u_{\alpha}$ is called the \emph{excess function} of $\alpha$, and it is easy to show that $\alpha \in PF_n$ if and only if $u_{\alpha}(j)\leq 0$, for all $j$. The set $U_{\alpha}=\{ 1\leq j\leq n\, |\, u_{\alpha}(j)\geq 1 \}$ is the \emph{excess set} of $\alpha$, and 
records those positions which prevent $\alpha$ from being a parking function (but not a Naples parking function). Notice that, since $\sum_{i=1}^{n}|\alpha|_i =n$, we have $u_{\alpha}(1)=0$, hence $1\notin U_{\alpha}$.

The next couple of lemmas (whose proofs are rather straightforward) record some basic properties of the excess function, and can be found in \cite{FV}.

\begin{lemma}\label{elementary_properties}
	Let $\alpha \in PP_n$. For all $1\leq j,j_1 ,j_2 \leq n$, with $j_1 <j_2$, we have: 
	\begin{itemize}
		\item[(i)] $u_\alpha (j)=j-1-\sum_{i=1}^{j-1}|\alpha |_i$;
		\item[(ii)] $u_\alpha (j_2 )-u_\alpha (j_1 )=j_2 -j_1 -\sum_{i=j_1 }^{j_2 -1}|\alpha |_i$;
		\item[(iii)] $u_\alpha (j)=u_\alpha (j+1)+|\alpha |_j -1$, when $j<n$;
		\item[(iv)] $u_\alpha (j)<j$.
	\end{itemize} 
\end{lemma}

\begin{lemma}\label{elementary_intervals}
	Let $\alpha \in PP_n$ and let $[p,q]$ be a maximal interval of consecutive positions contained in $U_\alpha$. Then we have: 
	\begin{itemize}
		\item[(i)] $u_\alpha (p)=1$;
		\item[(ii)] $u_\alpha (p-1)=0$;
		\item[(iii)] $|\alpha |_{p-1}=0$;
		\item[(iv)] $|\alpha |_q \geq 2$.
	\end{itemize}     
\end{lemma} 

The set of $k$-Naples parking functions cannot be characterized in terms of the excess function only, since the excess function does not depend on the order of the cars, however the following proposition (again to be found in \cite{FV}) gives some useful information.

\begin{proposition}\label{naples_excess}
	If $\alpha \in PF_{n,k}$, then $u_\alpha (j)\leq k$, for all $1\leq j\leq n$.
\end{proposition}



We conclude this section with a few basic definitions and notations concerning posets.

Given a poset $(P,\leq )$, a subset $F\subseteq P$ is called an \emph{up-set} when, for all $x\in F$ and $y\in P$, if $x\leq y$ then $y\in F$ (in other words, $F$ is closed upwards). Given a lattice $L$, an up-set $F$ of $L$ is said to be a \emph{filter} when, for all $x,y\in L$, $x\wedge y\in L$. For any subset $Q\subset L$, we denote with~$\uparrow \! Q$ the filter \emph{generated by} $Q$, which is by definition the smallest filter of $L$ containing~$Q$. In particular, when $Q=\{ x\}$, this is called the \emph{principal filter} generated by $x$, and will be simply denoted $\uparrow \! x$. If the lattice $L$ is finite, then every filter of $L$ is principal (of course this is no longer true for generic up-sets).

\section{$\rho$-Naples parking functions}\label{main}

Denote with $\Delta_n$ the set of all subdiagonal sequences of natural numbers of length $n$, i.e. $\Delta_n =\{ (r_1 ,r_2 ,\ldots ,r_n )\in \mathbb{N}^n \, |\, r_i <i, \textnormal{ for all $i\leq n$} \}$. The elements of $\Delta_n$ will be called \emph{parking strategies} (of length $n$).
Let $\alpha =(a_1,a_2 ,\ldots ,a_n )\in PP_n$ and $\rho =(r_1 ,r_2 ,\ldots ,r_n)\in \Delta_n$. We define the \emph{$\rho$-Naples parking rule} by imposing that car $c_i$ follows the $r_i$-Naples parking rule, for all $i\leq n$. We say that $\alpha$ is a \emph{$\rho$-Naples parking function} when all cars are able to park under the $\rho$-Naples parking rule. The set of $\rho$-Naples parking functions of length $n$ will be denoted~$PF_{n,\rho}$. For any parking preference $\alpha$, we denote with $\psi_{\rho}(\alpha )=(\psi_{\rho}^{(\alpha )}(1),\psi_{\rho}^{(\alpha )}(2),\ldots ,\psi_{\rho}^{(\alpha )}(n))$ the \emph{outcome map} of $\alpha$ with respect to the parking strategy $\rho$, which is defined by setting $\psi_{\rho}^{(\alpha )}(i)$ to be the actual spot in which $c_i$ parks following the $\rho$-Naples parking rule (with the convention that $\psi_{\rho}^{(\alpha )}(i)=\infty$ if $c_i$ is unable to park).

The notion of $\rho$-Naples parking function generalizes that of classical parking function, and more generally that of $k$-Naples parking function. Indeed, setting $\mathbf{0}_n=(\underbrace{0,\ldots ,0}_{n})$,
a $\mathbf{0}_n$-Naples parking function corresponds to a classical parking function, whereas, for $k<n$, if we set $\mathbf{k}_n =(0,1,\ldots ,k-1,k,k,\ldots k)$
, a $\mathbf{k}_n$-Naples parking function is precisely a $k$-Naples parking function. We also notice that every parking preference of length $n$ is a $\mathbf{(n-1)}_n$-Naples parking function.

\bigskip

As an example, consider the parking preference $\alpha =(4,3,3,4,1)$. We observe that $\alpha \notin PF_{5,\mathbf{0}_n} = PF_5$, whereas $\alpha \in PF_{5,\mathbf{k}_n}$, for all $0<k<n$. Moreover the reader can easily check that, for instance, $\alpha \notin PF_{5,11011}$, whereas $\alpha \in PF_{5,00100}$ (where, for simplicity, we have represented $\rho$ as the string of its entries; we will do the same in the rest of the paper, unless explicitly stated otherwise). More specifically, the outcome maps of $\alpha$ with respect to the two above parking rules are $\psi_{11011}(\alpha )=(4,3,5,\infty ,1)$ and $\psi_{00100}(\alpha )=(4,3,2,5,1)$, respectively. 

\section{The lattice of parking strategies}\label{lattice}

The next proposition is one of the key result of our paper. Suppose that $\alpha$ is a $\rho$-Naples parking function, with $\rho=r_1 r_2 \cdots r_n$, and let $\rho' =r'_1 r'_2 \cdots r'_n$ be such that $r_i \leq r'_i$, for all $i\leq n$. It looks natural to infer that we should have that $\alpha$ is also a $\rho'$-Naples parking function, since we are just allowing cars to check more backward spots. However, this is at least not trivial as it seems since, for instance, the outcome maps $\psi_{\rho}$ and $\psi_{\rho'}$ are not necessarily equal (for example, take again the above considered parking preference $\alpha =(4,3,3,4,1)$ and the two parking strategies $\rho =00120$ and $\rho' =00130$). Luckily, the next proposition shows that the above observation is not an impediment to the intuitive fact that the property of being a $\rho$-parking function is preserved by replacing $\rho$ with a componentwise larger parking strategy. As we will see in a moment, the proof is however not at all trivial. We give a rough illustration of the main argument of the proof in Figure \ref{up-set}.

\begin{proposition}\label{upset}
	Let $\alpha =(a_1 ,a_2 ,\ldots ,a_n ) \in PF_{n,\rho}$, with $\rho=r_1 r_2 \cdots r_n$, and let $\rho' =r'_1 r'_2 \cdots r'_n$ be such that $r_i \leq r'_i$, for all $i\leq n$. Then $\alpha \in PF_{n,\rho'}$.
\end{proposition}

\begin{proof}
	Since we have a single parking preference for a certain set of cars and two distinct parking strategies associated with it, we will briefly refer to the two (parking) problems as problem $\rho$ and problem $\rho'$.
	
	We claim that, for all $h\leq n$, if $\psi_{\rho'}^{(\alpha )}(h)\geq \psi_{\rho}^{(\alpha )}(h)$, then $\psi_{\rho'}^{(\alpha )}(h)\in \{ \psi_{\rho}^{(\alpha )}(1),\psi_{\rho}^{(\alpha )}(2),\ldots ,\psi_{\rho}^{(\alpha )}(h) \}$. Notice that this is stronger than what we actually need to prove. In fact, our goal is just to show that, for all $h\leq n$, $\psi_{\rho'}^{(\alpha )}(h)<\infty$, which is clearly implied by the above claim (since $\alpha \in PF_{n,\rho }$, and so $\infty \notin \{ \psi_{\rho}^{(\alpha )}(1),\psi_{\rho}^{(\alpha )}(2),\ldots ,\psi_{\rho}^{(\alpha )}(h) \}$).
	
	We argue by induction on $h$ and begin by observing that of course $\psi_{\rho'}^{(\alpha )}(1)=\psi_{\rho}^{(\alpha )}(1)$. Now assume that the claim holds for all $i<h$, and suppose that $\psi_{\rho'}^{(\alpha )}(h)\geq \psi_{\rho}^{(\alpha )}(h)$. If $\psi_{\rho'}^{(\alpha )}(h)=\psi_{\rho}^{(\alpha )}(h)$ the claim is clearly true. So suppose (for the rest of the proof) that $\psi_{\rho'}^{(\alpha )}(h)>\psi_{\rho}^{(\alpha )}(h)$. Equivalently, this means that in problem $\rho'$ car $c_h$ either parks 
	after the spot occupied by car $c_h$ in problem $\rho$, or it cannot park at all (i.e. $\psi_{\rho'}^{(\alpha )}(h)=\infty$).  
	
	If $a_h \geq \psi_{\rho'}^{(\alpha )}(h)$, then, in problem $\rho$, car $c_h$ must have checked spot $\psi_{\rho'}^{(\alpha )}(h)$ (by driving backwards from spot $a_h$ to reach spot $\psi_{\rho}^{(\alpha )}(h)$), finding it occupied. This means that $\psi_{\rho'}^{(\alpha )}(h)\in \{ \psi_{\rho}^{(\alpha )}(1),\psi_{\rho}^{(\alpha )}(2),\ldots ,\psi_{\rho}^{(\alpha )}(h) \}$, which is our claim.
	
	Otherwise, suppose that $a_h <\psi_{\rho'}^{(\alpha )}(h)$. We observe that, in problem $\rho'$, car $c_h$ must have checked spot $\psi_{\rho}^{(\alpha )}(h)$, finding it occupied. Indeed, if $a_h \leq \psi_{\rho}^{(\alpha )}(h)$, then, in problem $\rho'$, spot $\psi_{\rho}^{(\alpha )}(h)$ is between the preferred spot for $c_h$ (which is $a_h$) and the actual spot in which $c_h$ parks (which is $\psi_{\rho'}^{(\alpha )}(h)$), so $c_h$ (which clearly parks driving forward) has found spot $\psi_{\rho}^{(\alpha )}(h)$ already occupied. 	
	On the other hand, if $\psi_{\rho}^{(\alpha )}(h)<a_h <\psi_{\rho'}^{(\alpha )}(h)$, then in problem $\rho$, car $c_h$ parks in spot $\psi_{\rho}^{(\alpha )}(h)$ by driving backward from $a_h$. Since, in problem $\rho'$, car $c_h$ instead parks in $\psi_{\rho'}^{(\alpha )}(h)$ by driving forward from $a_h$, and we know that $r'_h \geq r_h$, then certainly $c_h$ has checked spot $\psi_{\rho}^{(\alpha )}(h)$ (when looking for a spot by driving backward), finding it occupied. Thus, in all cases, we can conclude that $\psi_{\rho}^{(\alpha )}(h)\in \{ \psi_{\rho'}^{(\alpha )}(1),\psi_{\rho'}^{(\alpha )}(2),\ldots ,\psi_{\rho'}^{(\alpha )}(h-1) \}$, i.e., there exists $j_1 <h$ such that $\psi_{\rho}^{(\alpha )}(h)=\psi_{\rho'}^{(\alpha )}(j_1 )$.
	
	Now observe that, since in problem $\rho$ car $c_h$ parks in $\psi_{\rho}^{(\alpha )}(h)=\psi_{\rho'}^{(\alpha )}(j_1 )$, then necessarily $\psi_{\rho'}^{(\alpha )}(j_1 )\notin \{ \psi_{\rho}^{(\alpha )}(1),\psi_{\rho}^{(\alpha )}(2),\ldots ,\psi_{\rho}^{(\alpha )}(h-1) \} \supseteq \{ \psi_{\rho}^{(\alpha )}(1),\psi_{\rho}^{(\alpha )}(2),\ldots ,\psi_{\rho}^{(\alpha )}(j_1 ) \}$, hence (by the inductive hypothesis) $\psi_{\rho'}^{(\alpha )}(j_1 )<\psi_{\rho}^{(\alpha )}(j_1 )$.
	
	If $\psi_{\rho}^{(\alpha )}(j_1 )=\psi_{\rho'}^{(\alpha )}(h)$, then clearly our claim follows.
	
	If $\psi_{\rho}^{(\alpha )}(j_1 )>\psi_{\rho'}^{(\alpha )}(h)$, we again distinguish two cases, depending on the relative positions of $a_{j_1}$ and $\psi_{\rho'}^{(\alpha )}(h)$. If $\psi_{\rho'}^{(\alpha )}(h)<a_{j_1}$, then we have that $\psi_{\rho '}^{(\alpha )}(j_1 )=\psi_{\rho}^{(\alpha )}(h)<\psi_{\rho '}^{(\alpha )}(h)<a_{j_1 }$, and so in problem $\rho'$ car $c_{j_1}$ would find spot $\psi_{\rho'}^{(\alpha )}(h)$ already occupied 
	(when driving backwards towards spot $\psi_{\rho'}^{(\alpha )}(j_1 )$), which is impossible, since $j_1 <h$. If instead $a_{j_1}\leq \psi_{\rho'}^{(\alpha )}(h)$, then in problem $\rho$ car $c_{j_1}$ finds spot $\psi_{\rho'}^{(\alpha )}(h)$ already occupied (when driving forward towards spot $\psi_{\rho}^{(\alpha )}(j_1 )$), hence $\psi_{\rho'}^{(\alpha )}(h)\in \{ \psi_{\rho}^{(\alpha )}(1),\psi_{\rho}^{(\alpha )}(2),\ldots ,\psi_{\rho}^{(\alpha )}(j_1 -1) \} \subseteq \{ \psi_{\rho}^{(\alpha )}(1),\psi_{\rho}^{(\alpha )}(2),\ldots ,\psi_{\rho}^{(\alpha )}(h) \}$, which is our claim.
	
	If $\psi_{\rho'}^{(\alpha )}(h)>\psi_{\rho}^{(\alpha )}(j_1 )$, then we can argue as above, with car $c_{j_1}$ playing the role of car $c_h$ in problem $\rho$ (specifically, we can repeat the whole argument starting from the assumption $\psi_{\rho'}^{(\alpha )}(h)>\psi_{\rho}^{(\alpha )}(h)$ made in the third paragraph of the present proof, and replacing $\psi_{\rho}^{(\alpha )}(h)$ with $\psi_{\rho}^{(\alpha )}(j_1 )$). In such a way, we get that, if $a_{j_1}\geq \psi_{\rho'}^{(\alpha )}(h)$, then our claim holds. Otherwise, i.e., if $a_{j_1}<\psi_{\rho'}^{(\alpha )}(h)$, then in problem $\rho'$ car $c_h$ must have checked spot $\psi_{\rho}^{(\alpha )}(j_1 )$, finding it occupied, hence $\psi_{\rho}^{(\alpha )}(j_1 )\in \{ \psi_{\rho'}^{(\alpha )}(1),\psi_{\rho'}^{(\alpha )}(2),\ldots ,\psi_{\rho'}^{(\alpha )}(h-1) \}$. In other words, there exists $j_2 <h$ such that $\psi_{\rho}^{(\alpha )}(j_1 )=\psi_{\rho'}^{(\alpha )}(j_2 )$.
	
	Now, since the total number of spots (and of cars) is finite, after a finite number of iterations of the previous argument we will find an index $k<h$ such that $\psi_{\rho'}^{(\alpha )}(h)\leq \psi_{\rho}^{(\alpha )}(k)$, which proves our claim (by referring to the appropriate part of the above argument).
\end{proof}

\begin{figure}[t]
	\begin{center}
		\includegraphics[scale=0.6]{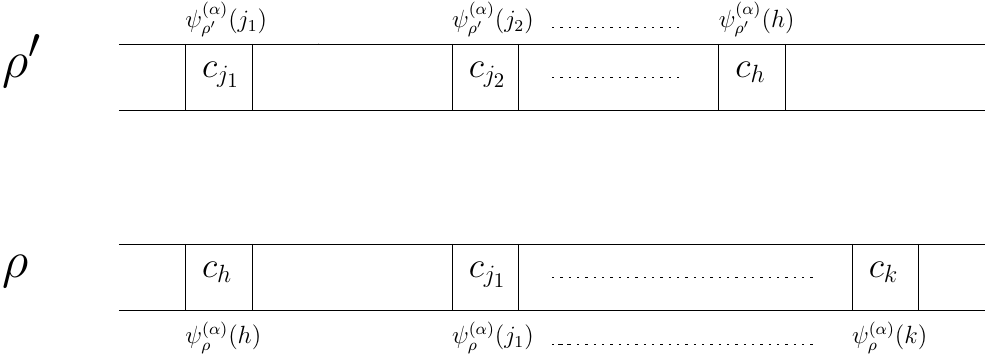}
		\caption{Notations are the same as in the proof of Proposition \ref{upset}. Two copies of the same road are displayed, illustrating the main recursive argument of the proof.}
		\label{up-set}
	\end{center}
\end{figure}

The result obtained in the above proposition can be naturally interpreted in an order-theoretic perspective. The set $\Delta_n \subseteq \mathbb{N}^n$ of parking strategies of length $n$ has a natural poset structure, which comes from taking the componentwise order of its elements. Specifically, given $\rho ,\rho' \in \Delta_n$, with $\rho=r_1 r_2 \cdots r_n$ and $\rho' =r'_1 r'_2 \cdots r'_n$, we set $\rho \leq \rho'$ whenever $r_i \leq r'_i$, for all $i\leq n$. Endowed with this partial order, the poset $\Delta_n$ is in fact a (finite) lattice, and more precisely it is the product of chains $\{ 0\}\times \{ 1\}\times \cdots \times \{ n-1\}$, where $\{ i\} =\{ 0,1,\ldots ,i\}$ 
with the natural total order. As such, the lattice $\Delta_n$ is distributive, hence ranked, and the rank of an element $\rho=r_1 r_2 \cdots r_n$ is given by $|\rho |=\sum_{i=1}^{n}r_i$. It is immediate to realize that the minimum of $\Delta_n$ is $\mathbf{0}_n$, whereas its maximum is $\mathbf{(n-1)}_n$. This lattice will be called the \emph{lattice of parking strategies} (of order $n$).

In terms of the poset structure of $\Delta_n$, Proposition \ref{upset} says that, given a parking preference $\alpha \in PP_n$, the set of all parking strategies $\rho$ such that $\alpha \in PF_{n,\rho}$ is an up-set $\mathcal{F}_{\alpha}$ of $\Delta_n$. In such a case, we will say that $\mathcal{F}_{\alpha}$ is \emph{realized by} $\alpha$, or that $\alpha$ \emph{induces} the up-set $\mathcal{F}_{\alpha}$. Moreover, an up-set $\mathcal{F}$ of the lattice of parking preferences is said to be a \emph{realizer} when there exists a parking preference $\alpha$ such that $\mathcal{F}=\mathcal{F}_{\alpha}$. In particular, it is clear that $\alpha \in PF_n$ if and only if $\mathcal{F}_{\alpha}=\Delta_n$. A natural and interesting question is then to understand which up-sets of $\Delta_n$ are actually realizers. This problem remains open in its full generality. Below we will give some partial results that begin to shed some light on this intriguing question.

Our first achievement is that at least the filters (i.e. principal up-sets) of $\Delta_n$ are realizers.

\begin{theorem}\label{filter}
	Let $\rho \in \Delta_n$. Then there exists $\alpha \in PP_n$ such that $\mathcal{F}_{\alpha}=\uparrow \! \rho$.  
\end{theorem}

\begin{proof}
	The proof is constructive, as we will explicitly describe a parking preference $\alpha$ satisfying the desired property. 
	
	Given $\rho =r_1 r_2 \cdots r_n \in \Delta_n$, define $a_i =n+1-i+r_i$, and set $\alpha =(a_1 ,a_2 ,\ldots ,a_n )$. We start by showing that $\alpha \in PF_{n,\rho}$ (i.e., $\rho \in \mathcal{F}_{\alpha}$). To this aim, we prove the stronger claim that $\psi_{\rho}^{(\alpha )}(i)=n+1-i$, for all $i\leq n$ (roughly speaking, this means that each car occupies the rightmost available spot). Indeed, we clearly have that $\psi_{\rho}^{(\alpha )}(1)=a_1 =n$ (since $r_1 =0$). Moreover, supposing inductively that cars $c_1,\ldots ,c_{i-1}$ have occupied all spots from $n-i+2$ to $n$ (i.e., the last $i-1$ spots), we focus on car $c_i$. If $r_i =0$, then $a_i =n+1-i$, so $c_i$ can park in its preferred spot (as it is free). Otherwise, if $r_i >0$, then necessarily $a_i =n+1-i+r_i >n+1-i$, which means that the preferred spot of car $c_i$ is already occupied. So car $c_i$ has to drive backwards and, after \emph{exactly} $r_i$ backward steps, it finds the free spot $n+1-i$. This completes the proof of the fact that $\psi_{\rho}^{(\alpha )}(i)=n+1-i$.
	
	As a consequence of Proposition \ref{upset}, we have that, if $\rho'=r'_1 r'_2 \cdots r'_n \geq \rho$, then $\alpha \in PF_{n,\rho'}$, that is $\uparrow \! \rho \subseteq \mathcal{F}_{\alpha}$. Conversely, suppose that $\rho' \ngeq \rho$. This means that there exists an entry of~$\rho'$ which is strictly smaller than the corresponding entry of $\rho$. Let $i$ be the smallest index such that $r'_i <r_i$. Since, for all $j<i$, we have $r'_j \geq r_j$, the first $i-1$ cars parks in the same spots both in problem $\rho$ and in problem~$\rho'$ (as, for each such $j$, car $c_j$ has to drive backwards just $r_j$ steps to find the first available spot). However, when car $c_i$ tries to park in problem~$\rho$, it needs to drive backward at least one step (since $r_i >r'_i \geq 0$) and we have seen in the above proof that~$r_i$ is the \emph{minimum} number of backward steps it has to take in order to find an available spot. This implies that, in the problem $\rho'$, car $c_i$ is not able to park (as $r_i >r'_i$), hence $\alpha \notin PF_{n,\rho'}$. So we have showed that $\mathcal{F}_{\alpha}\subseteq \uparrow \! \rho$, which concludes the proof.	
\end{proof}

We remark that, for a given strategy $\rho$, the parking preference $\alpha$ described in the above proof is not the unique preference such that $\mathcal{F}_{\alpha}=\uparrow \! \rho$. Consider, for instance, the parking strategy $001 \in \Delta_3$: the two preferences $\alpha =(3,2,2)$ and $\beta =(1,3,3)$ are both such that $\mathcal{F}_{\alpha}=\mathcal{F}_{\beta}=\uparrow \! 001$. In other words, there are (principal) realizers that can be induced by different parking preferences.

Moreover, Theorem \ref{filter} asserts that any filter in $\Delta_n$ is a realizer. However, this is not true for a generic (non-principal) up-set. For example, when $n=3$, consider the up-set $\{ 011,002,012 \}$ of $\Delta_3$, and suppose that there exists $\alpha=(a_1 ,a_2 ,a_3 )\in PP_3$ such that $\mathcal{F}_{\alpha}=\{ 011,002,012 \}$. If $a_1 \neq a_2$, then certainly cars $c_1 $ and $c_2$ park in their respective preferred spots, thus, since $\alpha \in PF_{3,011}$, we also get $\alpha \in PF_{3,001}$, which is impossible (as $001 \notin \mathcal{F}_{\alpha}$). Therefore, suppose that $a_1 =a_2$. If $a_1 =a_2 =1$, then $\alpha \in PF_3 =PF_{3,000}$, which is impossible. If $a_1 =a_2 =2$, then it is easy to see that, independently from the value of $a_3$, we have $\alpha \in P_{3,010}$, again a contradiction. Finally, if $a_1 =a_2 =3$, then $\alpha \notin PF_{3,002}$ (as car $c_2$ is not able to park), against the hypothesis.

On the other hand, there exist non-principal up-sets that are realizers. For instance, it is not hard to see that, in $\Delta_3$, the up-set $\uparrow \! \{ 001,010\}$ is induced by the parking preference $222$, i.e. $\mathcal{F}_{222}=\uparrow \! \{ 001,010\}$. This suggests the particularly intriguing problem of characterizing which up-sets are realizers. While we have not been able to give a complete answer, we will provide a partial one in the next section. 

\section{Minimizing the number of backward steps}\label{min_back_steps}

Given a parking preference $\alpha$, a complete understanding of $\mathcal{F}_{\alpha}$ can be provided by describing the antichain of its minimal elements. This is another extremely interesting problem, which is still open in general. In the present section we move the first steps towards a solution by illustrating a characterization of those strategies $\rho =r_1 r_2 \cdots r_n$ which minimize the total number $|\rho |=\sum_{i=1}^{n}r_i$ of allowed backward steps for a given $\alpha \in PP_n$ (while still allowing each car to succesfully park). These are indeed a particular class of minimal elements of $\mathcal{F}_{\alpha}$, namely those having minimal rank in the lattice of parking strategies of order $n$. We will briefly call each of such strategies a \emph{minimizer} for $\alpha$.  

Before starting, we would like to remark that, for a given parking preference $\alpha$, minimizing the number of allowed backward steps for a succesful strategy is \emph{not} equivalent to finding the actual number of backward steps taken by all cars when following the $\mathbf{(n-1)}_n$-Naples parking rule. This is due to the fact that increasing the number of backward steps allowed for a single car may enable to decrease the number of backward steps required for several other cars. 
As an example, consider $\alpha =(4,4,3,2,3)\in PP_5$. If we use the $\mathbf{(n-1)}_n$-Naples parking rule, then each of the cars $c_2 ,c_3$ and $c_4$ needs to take a backward step in order to successfully park, for a total of 3 backward steps. On the other hand, it is not difficult to see that $\alpha \in PF_{5,00002}$ (in fact, the parking strategy 00002 minimizes the total number of allowed backward steps for $\alpha$).

Our first result provides a recipe to effectively construct a minimizer for a preference $\alpha$ which is based on the notion of excess function. We first need a technical lemma, which was proved in \cite{FV} for $k$-Naples parking function, but translates almost verbatim to the case of $\rho$-Naples parking functions. At any rate, we include the proof for the sake of completeness. Recall from Section \ref{preliminaries} that $U_{\alpha}=\{ 1\leq j\leq n\, |\, u_{\alpha}(j)\geq 1 \}$.

\begin{lemma}\label{p-1}
	Given $\alpha \in PP_n$ and $\rho \neq \mathbf{0}_n$, we have that $\alpha$ is a $\rho$-Naples parking function if and only if, for every maximal interval $[p,q]\subseteq U_{\alpha}$, spot $p-1$ is occupied.
\end{lemma}


\begin{proof}
	Clearly, if $\alpha$ is a $\rho$-Naples parking function, then every spot is occupied by some car.
	
	On the other hand, suppose that $\alpha$ is not a $\rho$-Naples parking function. Let $h$ be the smallest spot which remains unoccupied. Then all spots in $[1,h-1]$ are occupied by cars having preferences in $[1,h-1]$, and there is no car having preference in $[1,h-1]$ which parks after spot $h$. We thus have (by $(i)$ of Lemma \ref{elementary_properties}) that $u_{\alpha}(h)=h-1-\sum_{i=1}^{h-1}|\alpha |_i=0$, i.e. $h\notin U_{\alpha}$. Moreover, since $h$ remains unoccupied, then $|\alpha |_h =0$, which gives (using $(iii)$ of Lemma \ref{elementary_properties}) $u_{\alpha}(h+1)=u_{\alpha}(h)-|\alpha |_h +1=1$, i.e. $h+1$ is the minimum of a (maximal) interval of $U_{\alpha}$.
\end{proof}

The theorem is preceded by another lemma which states an easy lower bound.

\begin{lemma}\label{ineq}
	Let $\rho \in \Delta_n$ and $\alpha \in PP_n$. If $\alpha \in PF_{n,\rho}$, then
	\[
		|\rho |\geq \sum_{j\in U_{\alpha}}u_{\alpha}(j).
	\]
\end{lemma}

\begin{proof}
	Since $\alpha \in PF_{n,\rho}$, then all cars are able to park using the $\rho$-Naples parking rule. Let $j\in U_{\alpha}$. Since the cars having preference at least $j$ which park in the spots from $j$ to $n$ are at most $n-j+1$, by definition of the excess function we have that there are at least $u_{\alpha}(j)\geq 1$ cars (among those having preference at least $j$) which park in the first $j-1$ spots. This implies that at least $u_{\alpha}(j)$ cars take at least one backward step from $j$ to $j-1$. Summing over all $j\in U_{\alpha}$ gives the desired lower bound.  
\end{proof}

%

\begin{theorem}\label{minimize_back_steps}
	Let $\alpha \in PP_n$. For each maximal interval $[p,q]\subseteq U_{\alpha}$, and for each $j\in [p,q]$, define the set $T_j$ as follows:
	\begin{itemize}
		\item if $j=q$, then $T_q$ is the set of the $u_{\alpha}(q)$ largest indices $i$ in $[1,n]$ such that $a_i =q$;
		\item if $j\in [p,q-1]$, then $T_j$ is the set of the $u_{\alpha}(j)$ largest indices $i$ in $[1,n]$ such that $a_i =j$ or $i\in T_{j+1}$. 
	\end{itemize}
	
	Let $\rho$ be the parking strategy allowing a car $c_i$ to drive backwards precisely as many spots as the number of times $i$ is contained in some set $T_j$, i.e.,
	\[
		\rho =\sum_{j\in U_{\alpha}}\sum_{i\in T_j}e_i ,
	\]
	where $e_i \in \Delta_n$ is the $n$-tuple all of whose components are 0 except for the $i$-th one which is 1, and the sum is taken componentwise. Then $\alpha \in PF_{n,\rho}$ and $|\rho |=\sum_{j\in U_{\alpha}}u_{\alpha}(j)$, i.e. $\rho$ is a minimizer for $\alpha$.
\end{theorem}

\begin{proof}
	We start by showing that the sets $T_j$ are well defined. Given a maximal interval $[p,q]$ of $U_{\alpha}$, first of all we observe that $q+1\notin U_{\alpha}$, hence $u_{\alpha}(q+1)\leq 0$, and so (using (iii) of Lemma \ref{elementary_properties}) $|\alpha |_q =u_{\alpha}(q)-u_{\alpha}(q+1)+1\geq u_{\alpha}(q)+1$. This argument holds for $q\neq n$, whereas when $q=n$, we get that $|\alpha |_n =u_{\alpha}(n)+1$ by (i) of Lemma \ref{elementary_properties}. Therefore, in any case, $T_q$ is well defined, i.e. we can find \emph{at least} $u_{\alpha}(q)$ indices with the required properties to belong to $T_q$. We now argue by induction, and suppose that, for a given $j<q$, sets $T_i$ are defined and contain $u_{\alpha}(i)$ indices for all $j+1\leq i\leq q$. We thus have that the number of indices $h$ for which $a_h=j$ or $h\in T_{j+1}$ is given by $|\alpha |_j +u_{\alpha}(j+1)=u_{\alpha}(j)+1$ (since the two conditions are clearly mutually exclusive), hence $T_j$ is well defined and contains $u_{\alpha}(j)$ indices. Moreover, since each element of each $T_j$ (for all $j\in U_{\alpha}$) contributes exactly one step to the parking strategy $\rho$ (by construction), then $|\rho |=\sum_{j\in U_{\alpha}}u_{\alpha}(j)$.
	
	Now let $[p,q]\subseteq U_{\alpha}$ be maximal, and consider all cars having preference $q$. Since $|\alpha |_q \geq u_{\alpha}(q)+1$ (by Lemma \ref{elementary_properties}), after the first of such cars attempts to park, spot $q$ is certainly occupied, and we are left with at least $u_{\alpha}(q)$ cars with preference $q$ still having to park. Therefore each car $c_i$ such that $i\in T_q$ has preference $q$ and it has a possible backward step to take (provided by the $n$-tuple $e_i$). This implies that all cars having indices in $T_q$ will reach spot $q-1$ (remember that $T_q$ contains the $u_{\alpha}(q)$ \emph{largest} indices $i$ such that $a_i =q$). Now suppose (once again by induction) that, given $j\in [p,q-1]$, there are exactly $u_{\alpha}(j+1)$ cars reaching spot $j$ by using backward steps provided by the sets $T_{j+1},\ldots ,T_q$. As a consequence, there are exactly $|\alpha|_j +u_{\alpha}(j+1)$ cars that either have preference $j$ or reach spot $j$ by using backwards steps provided by sets $T_{j+1},\ldots ,T_q$: the first of such cars will occupy spot $j$ (if that spot is not already occupied), whereas all the remaining $|\alpha|_j +u_{\alpha}(j+1)-1=u_{\alpha}(j)$ cars will reach spot $j-1$ (thanks to the backward steps provided by $T_j$).
	
	The above inductive argument holds for all indices in $[p,q]$. In particular, considering the index $p$, since $u_{\alpha}(p)=1$ (by (i) of Lemma \ref{elementary_intervals}), there is at least one car reaching spot $p-1$, which implies that spot $p-1$ is certainly occupied (by that car or some previous one). Since this holds for every maximal interval of $U_{\alpha}$, we can invoke Lemma \ref{p-1} to conclude that $\alpha \in PF_{n,\rho}$.   
\end{proof}

In order to illustrate the above theorem, we give an example. Let $\alpha =(4,9,5,9,5,8,7,9,4,6)\in PP_{10}$. First of all we compute the excess set of $\alpha$, which is $U_{\alpha}=[2,9]$. Therefore in this case $U_{\alpha}$ consists of a single maximal interval. The table below shows how to construct the sets $T_j$ associated with this interval:

\begin{center}
	\begin{tabular}{ccccc}
		\hline
		& $u_{\alpha}(j)$ & $\{i\, |\, a_i =j\}$ & $\{i\, |\, a_i =j\} \cup T_{j+1}$ & $T_j$ \\
		\hline
		$j=9$ & 1 & $\{ 2,4,8\}$ & $\{ 2,4,8\}$ & $\{ 8\}$ \\
		$j=8$ & 1 & $\{ 6\}$ & $\{ 6,8\}$ & $\{ 8\}$ \\
		$j=7$ & 1 & $\{ 7\}$ & $\{ 7,8\}$ & $\{ 8\}$ \\
		$j=6$ & 1 & $\{ 10\}$ & $\{ 8,10\}$ & $\{ 10\}$ \\
		$j=5$ & 2 & $\{ 3,5\}$ & $\{ 3,5,10\}$ & $\{ 5,10\}$ \\
		$j=4$ & 3 & $\{ 1,9\}$ & $\{ 1,5,9,10\}$ & $\{ 5,9,10\}$ \\
		$j=3$ & 2 & $\emptyset$ & $\{ 5,9,10\}$ & $\{ 9,10\}$ \\
		$j=2$ & 1 & $\emptyset$ & $\{ 9,10\}$ & $\{ 10\}$ \\
		\hline
	\end{tabular}
\end{center}

Applying Theorem \ref{minimize_back_steps}, we thus get the strategy $\rho =2e_5 +3e_8 +2e_9 +5e_{10}=0000200325$, and it is easy to check that indeed $\alpha \in PF_{10,\rho}$.

\bigskip

At this point it is natural to ask if the construction provided by Theorem \ref{minimize_back_steps} is unique. The answer is negative, i.e., for a given parking preference $\alpha$, there are strategies minimizing the number of backward steps other than the single one described above. The next theorem shows even more, as it completely characterizes \emph{all} minimizers. As it is evident from the statement, all such strategies have a construction similar to the one of Theorem \ref{minimize_back_steps}, with a slight modification concerning the choice of the sets $T_q$, where $q$ is the maximum of a maximal interval of $U_{\alpha}$.

\begin{theorem}
	Let $\alpha \in PP_n$. A parking strategy $\rho$ is a minimizer for $\alpha$ 
	if and only if, for every maximal interval $[p,q]\subseteq U_{\alpha}$ and for all $j\in [p,q]$, there exist sets $T_j$ for which $\rho =\sum_{j\in U_{\alpha}}\sum_{i\in T_j}e_i$, which are defined as follows:
	\begin{itemize}
		\item if $j=q$, then $T_q$ is any set of $u_{\alpha}(q)$ distinct indices $i$ in $[1,n]$ such that $a_i =q$, excluding the first one of such indices;
		\item if $j\in [p,q-1]$, then $T_j$ is the set of the $u_{\alpha}(j)$ largest indices $i$ in $[1,n]$ such that $a_i =j$ or $i\in T_{j+1}$. 
	\end{itemize}
\end{theorem}

\begin{proof}
	If sets $T_j$ and $\rho$ are constructed as in the statement, then an argument completely similar to that of Theorem \ref{minimize_back_steps} shows that $\alpha \in PF_{n,\rho}$ and $|\rho |=\sum_{j\in U_{\alpha}}u_{\alpha}(j)$.
	
	Conversely, suppose that $\rho$ is a minimizer for $\alpha$, i.e. $\alpha\ \in PF_{n,\rho}$ and $|\rho |=\sum_{j\in U_{\alpha}}u_{\alpha}(j)$. Given any $j\in U_{\alpha}$, by definition of the excess function, there are at least $u_{\alpha}(j)$ cars having preference at least $j$ that park in the first $j-1$ spots. Therefore, since each of those cars needs to take the backward step from spot~$j$ to spot $j-1$, and we know that $|\rho |=\sum_{j\in U_{\alpha}}u_{\alpha}(j)$, there must be \emph{exactly} $u_{\alpha}(j)$ such cars. In other words, every index $j\in U_{\alpha}$ contributes exactly $u_{\alpha}(j)$ to the components of the strategy $\rho$. 	
	
	Now let $[p,q]$ be a maximal interval of $U_{\alpha}$. From (ii) of Lemma \ref{elementary_properties} we know that $u_{\alpha}(p-1)=0$, hence $\sum_{i=p-1}^{n}|\alpha|_i =n-p+2$. If there were a car with preference strictly less than $p-1$ parking in one of the spots of the interval $[p-1,n]$, then there would be a car with preference greater than or equal to $p-1$ which would park in one of the first $p-2$ spots. Clearly, at some point, that car would have to take a backward step from $p-1$ to $p-2$, but the strategy $\rho$ does not permit it (since, by minimality, all the backward steps of $\rho$ are provided by the indices of $U_{\alpha}$, as shown in the above paragraph). Therefore, in particular, all spots in $[p-1,q]$ are occupied by cars having preferences in $[p-1,n]$.
	
	We now focus on spot $p-1$. Since (by Lemma \ref{elementary_intervals}) $|\alpha |_{p-1}=0$, spot $p-1$ is occupied by some car driving backwards from spot $p$. As seen above, the total number of cars trying to park in spot $p$ is given by the number of cars driving backwards from spot $p+1$, which is $u_{\alpha}(p+1)$, plus the number of cars having preference $p$, which is $|\alpha |_p$. One of such cars parks in spot $p$, and so each of the remaining $u_{\alpha}(p+1)+|\alpha |_p -1=u_{\alpha}(p)$ cars certainly takes the backward step from $p$ to $p-1$. Therefore $\rho$ must provide a backward step for each of the last $u_{\alpha}(p)$ cars reaching spot $p$. We can use a completely analogous argument for each $j\in [p,q-1]$, thus showing that $\rho$ must provide a backward step for each of the last $u_{\alpha}(j)$ cars reaching spot $j$. The situation is slightly different for spot $q$. Indeed, recalling that the backward steps recorded in $\rho$ are entirely provided by cars moving from spot $j$ to spot $j-1$ for all $j\in U_{\alpha}$, there are no cars taking the backward step from spot $q+1$ to spot $q$ (as $q+1\notin U_{\alpha}$). Moreover, since $u_{\alpha}(q+1)\leq 0$, we have (Lemma \ref{elementary_intervals})
	\[
		|\alpha |_q =u_{q}-u_{\alpha}(q+1)+1\geq u_{\alpha}(q)+1.
	\]  
 	 
 	 Thus, the $u_{\alpha}(q)$ cars taking the backward step from $q$ to $q-1$ can be arbitrarily chosen among the $|\alpha |_q$ cars having preference $q$ (we have enough of them as we just proved), except for the first one of them (which occupies spot $q$).
 	 
 	 Summing up what we have done, we can conclude that $\rho$ needs to be constructed exactly as described in Theorem \ref{minimize_back_steps}, with the only difference regarding the contribution of the last spots of each maximal interval of $U_{\alpha}$ to the components of $\rho$, as desired. 
\end{proof}

As a byproduct, for a given $\alpha \in PP_n$, we are able to count the number of parking strategies $\rho$ minizing the number of backward steps and such that $\alpha \in PF_{n,\rho}$.  

\begin{corollary}\label{count_minim}
	Let $\alpha \in PP_n$ and suppose that $U_{\alpha}=\bigcup_{i=1}^{m}[p_i ,q_i ]$ is the decomposition of $U_{\alpha}$ into maximal intervals. The number of minimizers for $\alpha$ 
	is
	\[
		\prod_{i=1}^{m}\binom{|\alpha |_{q_i}-1}{u_{\alpha}(q_i)} .
	\]
\end{corollary}

\begin{proof}
	Each of the required parking strategies corresponds to choosing the set $T_{q_j}$, for each $j=1,\ldots ,m$, and each such choice gives rise to a different parking strategy. Since each set $T_{q_j}$ consists of $u_{\alpha}(q_i)$ elements chosen among $|\alpha |_{q_i}-1$, the thesis immediately follows.  
\end{proof}

We now revisit the example provided after Theorem \ref{minimize_back_steps} in order to illustrate our last result. In that example we showed that the parking preference $\alpha =(4,9,5,9,5,8,7,9,4,6)\in PP_{10}$ is a $\rho$-Naples parking function, with $\rho =0000200325$. However this is not the only parking strategy that works. Indeed, since $|\alpha |_9 =3>2=u_{\alpha}(9)+1$, the above corollary shows that there is exactly another strategy $\tilde{\rho}$ such that $\alpha \in PF_{n,\tilde{\rho}}$ which minimizes the number of backward steps. Such a strategy can be explicitly contructed by choosing $T_9 =\{ 4\}$ instead of $T_9 =\{ 8\}$. Such a choice of course affects the construction of all the $T_j$'s, as described in the following table:

\begin{center}
	\begin{tabular}{ccccc}
		\hline
		& $u_{\alpha}(j)$ & $\{i\, |\, a_i =j\}$ & $\{i\, |\, a_i =j\} \cup T_{j+1}$ & $T_j$ \\
		\hline
		$j=9$ & 1 & $\{ 2,4,8\}$ & $\{ 2,4,8\}$ & $\{ 4\}$ \\
		$j=8$ & 1 & $\{ 6\}$ & $\{ 4,6\}$ & $\{ 6\}$ \\
		$j=7$ & 1 & $\{ 7\}$ & $\{ 6,7\}$ & $\{ 7\}$ \\
		$j=6$ & 1 & $\{ 10\}$ & $\{ 7,10\}$ & $\{ 10\}$ \\
		$j=5$ & 2 & $\{ 3,5\}$ & $\{ 3,5,10\}$ & $\{ 5,10\}$ \\
		$j=4$ & 3 & $\{ 1,9\}$ & $\{ 1,5,9,10\}$ & $\{ 5,9,10\}$ \\
		$j=3$ & 2 & $\emptyset$ & $\{ 5,9,10\}$ & $\{ 9,10\}$ \\
		$j=2$ & 1 & $\emptyset$ & $\{ 9,10\}$ & $\{ 10\}$ \\
		\hline
	\end{tabular}
\end{center}

The parking strategy $\tilde{\rho}$ arising from the above costruction is thus $\tilde{\rho}=0001211025$.

\bigskip

An interesting consequence of Corollary \ref{count_minim} is the following proposition, which identifies a class of up-sets of $\Delta_n$ which are not realizers. This result generalizes the example of a non-principal up-set which is not a realizer given after Theorem \ref{filter}.

\begin{proposition}
	Let $n\geq 3$, and suppose that $\mathcal{F}$ is an up-set of $\Delta_n$ generated by (at least two) elements having rank $\binom{n}{2}-1$ (i.e. coatoms of $\Delta_n$). Then $\mathcal{F}$ is not a realizer.
\end{proposition} 

\begin{proof}
	By contradiction, suppose that there exists $\alpha \in PP_n$ such that $\mathcal{F}=\mathcal{F}_{\alpha}$. Given a minimal element $\rho \in \mathcal{F}$, which has rank $\binom{n}{2}-1$, we have that
	\begin{equation}\label{formula}
		\frac{n(n-1)}{2}-1=|\rho |=\sum_{j\in U_{\alpha}}u_{\alpha}(j).
	\end{equation}
	We now focus on $u_{\alpha}(n)$. If $u_{\alpha}(n)=n-1$, then by definiton of excess function we have that $|\alpha |_n=u_{\alpha}(n)+1=n$, hence $\alpha =(n,n,\ldots ,n)$. As a consequence, for all $j\in [2,n]$, we get
	\[
		u_{\alpha}(j)=\sum_{i=j}^{n}|\alpha |_i -(n-j+1)=|\alpha |_n -(n-j+1)=j-1.
	\]
	
	This would imply that
	\[
	\sum_{j\in U_{\alpha}}u_{\alpha}(j)=\sum_{j=2}^{n}(j-1)=\frac{n(n-1)}{2},
	\]
	which contradicts (\ref{formula}).
	
	Therefore we must have $u_{\alpha}(n)\leq n-2$. By (iv) of Lemma \ref{elementary_properties}, we also have that $u_{\alpha}(j)\leq j-1$, for all $j\in [2,n-1]$. Thus, in order to have (\ref{formula}), we must impose that $u_{\alpha}(j)=j-1$ for all $j\in [2,n-1]$ and $u_{\alpha}(n)=n-2$. This means that $U_{\alpha}$ consists of a single maximal interval (i.e. $U_{\alpha}=[2,n]$) and $|\alpha |_n =u_{\alpha}(n)+1=n-1$. By Corollary \ref{count_minim}, the number of parking strategies $\rho \in \Delta_n$ such that $\alpha \in PF_{n,\rho}$ and $|\rho |=\frac{n(n-1)}{2}-1$ is given by $\binom{|\alpha |_n -1}{u_{\alpha}(n)}=1$. However, we are assuming that $\mathcal{F}$ has at least two elements of minimal rank, which gives again a contradiction.  
\end{proof}

We close this section with an apparently nontrivial observation concerning the total number of steps in a fixed parking preference. Given $\rho \in \Delta_n$ and $\alpha \in PP_n$, consider the set of all steps taken by all cars during the parking procedure according to $\alpha$. We say that a step is \emph{trivial} when it is taken by a car \emph{before} reaching its preferred spot. Thus every backward step is nontrivial, whereas the nontrivial forward steps are precisely those taken by some cars after having checked its preferred spot (and, if required, after having checked a certain number of previous spots). The next proposition states that a strategy that minimizes the number of backward steps also minimizes the total number of nontrivial steps, and this number relates to $\alpha$ in a very neat way. Recall that a \emph{minimizer} $\rho$ for a parking preference $\alpha$ is a parking strategy such that $\alpha \in PF_{n,\rho}$ and the number of backward steps taken during the parking process is minimal. 

\begin{proposition}
	Let $\alpha  =(a_1 ,a_2 ,\ldots ,a_n )\in PP_n$ and $\rho =r_1 r_2 \cdots r_n \in \Delta_n$ any parking strategy. The minimum total number of nontrivial steps needed for the parking process with respect to $\rho$ is $\sum_{j=1}^n |u_{\alpha}(j)|$, and this number is achieved if and only if $\rho$ is a minimizer.   
\end{proposition}

\begin{proof}
	Suppose for the moment that $\rho \in \mathcal{F}_{\alpha}$ is \emph{any} successful parking strategy for $\alpha$. Without loss of generality, we can suppose that $\rho$ is minimal in $\mathcal{F}_{\alpha}$. Indeed, if $\rho$ were not minimal, for all $i$ such that $r_i >0$, we could just replace $r_i$ with the exact number of backward steps taken by $c_i$ to park in the same spot: this does not change the outcome map of the preference, and so also the total number of nontrivial steps taken by all cars remains the same. 

	We start by computing the total number of nontrivial forward steps taken by all cars during the parking process, according to the parking preference $\alpha$ and following the parking strategy $\rho$. To this aim, we observe that, since $\rho$ is minimal in $\mathcal{F}_{\alpha}$, for all $r_i >0$ we have that $\psi_{\rho}^{(\alpha)} (i)=a_i -r_i$ (recall that $\psi_{\rho}^{(\alpha)}$ is the outcome map of $\alpha$ with respect to the parking strategy $\rho$), whereas all the remaining cars only take forward steps. Indeed, if there were a car $c_i$ parking in a spot larger than $a_i -r_i$, we could just replace $r_i$ with a smaller value (still having a valid parking strategy for $\alpha$), against the minimality of $\rho$. 

	Define $\tilde{\alpha}=(\tilde{a}_1 ,\tilde{a}_2 ,\ldots ,\tilde{a}_n)\in PP_n$ by setting $\tilde{a}_i =a_i -r_i$, for all $i$. 
	Then clearly $\tilde{\alpha}\in PF_n$, and the total number of nontrivial forward steps taken by all cars when parking according to $\tilde{\alpha}$ (using the classical parking rule) is equal to the total number of nontrivial forward steps taken by all cars when parking according to $\alpha$ (using the $\rho$-Naples parking rule). Now, focussing on $\tilde{\alpha}$, we observe that the total number of nontrivial forward steps taken by all cars can be computed as the difference between the sum of the actual spots in which each car parks and the sum of the preferred spots of each car. Since the image of the outcome map $\psi_{\rho}^{(\tilde{\alpha})}$ is a permutation of $[1,n]$ (because $\tilde{\alpha}\in PF_n$), we get that the total number of nontrivial forward steps taken by all cars when parking according to $\alpha$ using the $\rho$-Naples parking rule is given by
	\begin{align*}
		\sum_{i=1}^{n}i-\sum_{i=1}^{n}\tilde{a}_i&=\binom{n+1}{2}-\left( \sum_{i=1}^{n}a_i-\sum_{i=1}^{n}r_i \right) \\
		&=\binom{n+1}{2}-\sum_{j=1}^{n}j|\alpha |_j +|\rho | \\
		&\geq \binom{n+1}{2}-\sum_{j=1}^{n}j|\alpha |_j +\sum_{j\in U_{\alpha}}u_{\alpha}(j) \\
		&=\binom{n+1}{2}-\sum_{j=1}^{n}j|\alpha |_j +\sum_{j=1}^{n}\left( \sum_{k=j}^{n}|\alpha |_k -(n-j+1)\right) -\sum_{j\notin U_{\alpha}}u_{\alpha}(j) \\
		&=\binom{n+1}{2}-\sum_{j=1}^{n}j|\alpha |_j +\sum_{j=1}^{n}j|\alpha |_j -\sum_{j=1}^{n}(n-j+1)-\sum_{j\notin U_{\alpha}}u_{\alpha}(j) \\
		&=-\sum_{j\notin U_{\alpha}}u_{\alpha}(j),	
	\end{align*}
	where the inequality $|\rho |\ge \sum_{j\in U_{\alpha}}u_{\alpha}(j)$ follows from Lemma \ref{ineq}. Note, in particular, that the final quantity above is indeed nonnegative, since $u_{\alpha}(j)\le 0$ for all $j\notin U_{\alpha}$, and that it constitutes a lower bound depending only on $\alpha$ for the number of nontrivial forward steps. By Lemma \ref{ineq} and Theorem \ref{minimize_back_steps}, we get an equality in the above formula if and only if $\rho$ is a minimizer. Thus, $\rho$ minimizes the number of nontrivial forward steps if and only if $\rho$ minimizes the number of backward steps. In such a case, $\rho$ also minimizes the \emph{total} number of nontrivial steps, for a total of $\sum_{j=1}^n |u_{\alpha}(j)|$ steps, as desired.
%
\end{proof}

As a final remark, we observe that, since the excess function of a parking preference $\alpha$ is permutation-invariant (i.e., the excess function is the same for any rearrangement of the entries of $\alpha$), the minimum number of nontrivial steps for a successful parking strategy is also permutation-invariant (i.e. it remains the same for any rearrangement of the entries of $\alpha$).
	
\section{Minimizing the number of cars driving backwards}\label{min_cars_back}

Given a parking strategy $\rho$, the \emph{support} of $\rho$, denoted $\textrm{supp}(\rho )$, is the set of indices $i$ such that $r_i \neq 0$. Another step towards the understanding of the up-sets $\mathcal{F}_{\alpha}$ is the characterization of its minimal elements $\rho$ having minimal support. In the present section we move towards this direction, by investigating, for a given parking preference $\alpha$, the set of successful parking strategies for $\alpha$ having minimal support. In other words, we look for those successful parking strategies for $\alpha$ that minimize the number of cars driving backwards.   

The problem of minimizing the number of cars driving backwards is clearly different from the problem (discussed in the previous section) of minimizing the number of backward steps. As an example, given the parking preference $(3,3,2,4)$, the (successful) strategy $0110$ minimizes the number of backward steps, whereas the (successful) strategy $0003$ minimizes the support.

The next theorem is the main result of this section. It states a lower bound for the support of parking strategies in which those cars that are allowed to move backwards can always reach an unoccupied position. The subsequent proposition proves that such a lower bound can indeed be attained. We remark that the lower bound found in the next theorem corresponds to the intuitive strategy 
for which a car can move backwards only if it cannot park by moving forward. Though this is the most natural and intuitive strategy, we have not been able to find an equally simple proof of the corresponding lower bound. 


\begin{theorem}\label{lower_bound}
	Let $\alpha \in PP_n$. For each $j\in U_{\alpha}$, define $\tilde{T}_j$ as the set of the last $u_{\alpha}(j)$ cars having preference at least $j$, and set $\tilde{T}=\bigcup_{j\in U_{\alpha}}\tilde{T}_j$. Let $\rho =r_1 r_2 \cdots r_n \in \Delta_n$. 
	If $\alpha \in PF_{n,\rho}$, then $|\emph{supp}(\rho )|\geq |\tilde{T}|$.	
\end{theorem}
  
  
\begin{proof}
	The proof is a very careful and detailed analysis consisting of a rather large number of cases to consider. In the hope of making it more readable, we will mark each (nontrivial) case that will be dealt with by writing the first words of the corresponding paragraph in blue. This certainly does not simplify the proof itself, but hopefully might help the reader to parse it more efficiently. 
	
	Since we are interested only in the support of $\rho$, in the proof we will assume (without loss of generality) that for all $i\leq n$, either $r_i =0$ or $r_i =i-1$. 
	
	For all $i\leq n$, set
	\[
	\lambda_i =\min \{ j\in U_{\alpha}\, |\, \tilde{T}_j\cap [1,i]\neq \emptyset \} .
	\]
	
	By convention, if there is no $\tilde{T}_j$ such that $\tilde{T}_j\cap [1,i]\neq \emptyset$ for a given $i$, we set $\lambda_i =\infty$. In words, for a given preference $\alpha$, $\lambda_i$ denotes the minimum ``critical position" in $U_{\alpha}$ such that, among the first $i$ cars, there is at least one for which there is a positive excess of cars in the position corresponding to its preference. Directly from the definition, it is immediate to observe that $\lambda_{i+1}\leq \lambda_i$ , for all $i$.
	
	The key step for the proof consists of proving (by induction) the following statement: for all $i\leq n$,
	
	\begin{itemize}
			\item[(a)] if the first $i$ cars successfully park, then
			\begin{equation}\label{key}
				|\{ j\leq i\, |\, a_j \geq \lambda_i , r_j =j-1\} |\geq |\tilde{T}\cap [1,i]|;
			\end{equation}
			\item[(b)] moreover, if $\lambda_i \neq \infty$ and in (\ref{key}) equality holds, then all spots in $[\lambda_i , n]$ are occupied by some of the first $i$ cars.
	\end{itemize}

	The starting case of our inductive argument is when $i=1$. In such a case, clearly car $c_1$ (the first car) parks in its preferred spot $a_1$. Moreover, for every $j\in U_{\alpha}$, the set of spots $[j,n]$ is nonempty, thus $c_1$ cannot be among the last $u_{\alpha}(j)$ cars (i.e., one of the cars ``in excess") which tries to park in or after spot $j$ (because $c_1$ is the first car). 
	This means that, for all $j\in U_{\alpha}$, $1\notin \tilde{T}_j$, hence $1\notin \tilde{T}$, and so (\ref{key}) is trivially satisfied. Notice that statement (b) of our claim is also trivially satisfied, since $\lambda_1 =\infty$.
	
	{\color{blue}Now suppose that $2\leq i\leq n$}, and assume that the first $i$ cars park successfully. Moreover, suppose by induction that our claim holds for all values between 1 and $i-1$. We will prove our claim  by means of a careful and detailed analysis of all cases that can occur.
	
	If $i\notin \tilde{T}$, then necessarily $\lambda_i =\lambda_{i-1}$, and we immediately get (using the inductive hypothesis)
	\begin{align*}
	|\{ j\leq i\, |\, a_j \geq \lambda_i , r_j =j-1\} | &\geq |\{ j\leq i-1\, |\, a_j \geq \lambda_{i-1} , r_j =j-1\} | \\
														&\geq |\tilde{T}\cap [1,i-1]| \\
														&=|\tilde{T}\cap [1,i]|.
	\end{align*}
	
	Moreover, if $\lambda_i \neq \infty$ and in (\ref{key}) equality holds, then the above chain of inequalities becomes a chain of equalities, thus in particular $|\{ j\leq i-1\, |\, a_j \geq \lambda_{i-1} , r_j =j-1\} |=|\tilde{T}\cap [1,i-1]|$. By induction, we then have that all spots in $[\lambda_{i-1},n]=[\lambda_i ,n]$ are occupied by some of the first $i-1$ cars (and so by some of the first $i$ cars).
	
	{\color{blue}Now assume that $i\in \tilde{T}$.} First of all, we claim that $i\in \tilde{T}_{\lambda_i}$. Indeed, we have that $i\in \tilde{T}_j$, for some $j\in U_{\alpha}$, and so, by the definitions of $\lambda_i$ and $\tilde{T}_j$, we have that $a_i \geq j\geq \lambda_i$. In particular, note that $\lambda_i\neq\infty$. If $i\notin \tilde{T}_{\lambda_i}$, since $\tilde{T}_{\lambda_i}\cap [1,i]\neq \emptyset$, there must exist $h\leq i-1$ such that $h\in \tilde{T}_{\lambda_i}$. This means that car $c_h$ is one of the last $u_{\alpha}(\lambda_i )$ cars having preference at least $\lambda_i$. This yields a contradiction, since $c_i$ has preference $a_i \geq \lambda_i$ but $c_i$ is not among the last $u_{\alpha}(\lambda_i )$ cars with this property (since we assumed $i\notin \tilde{T}_{\lambda_i}$), even though $i>h$.
	
	We now first consider the case in which all spots in $[\lambda_i ,n]$ have been occupied after the first $i-1$ cars have parked. In this specific case, the fact that car $c_i$ is able to park forces $r_i =i-1$ (since, as we have shown above, $a_i \geq \lambda_i$). As a consequence, using the induction hypothesis, we get
	\begin{align*}
	|\{ j\leq i\, |\, a_j \geq \lambda_i , r_j =j-1\} |&=1+|\{ j\leq i-1\, |\, a_j \geq \lambda_i , r_j =j-1\} | \\
	&\geq 1+|\{ j\leq i-1\, |\, a_j \geq \lambda_{i-1} , r_j =j-1\} | \\
	&\geq 1+ |\tilde{T}\cap [1,i-1]| \\
	&=|\tilde{T}\cap [1,i]|.
	\end{align*}    
	Moreover, the second statement of our claim trivially holds for $i$.
	
	On the other hand, {\color{blue}suppose that there exists a spot $h\in [\lambda_i ,n]$ which remains unoccupied after the first $i-1$ cars have parked.} Again, we split our argument into two cases.
	
	First, assume that $|\tilde{T}\cap [1,i-1]|=0$. 
	Notice that, in the present case, we have that $i=\min \tilde{T}_{\lambda_i}$ (i.e., $c_i$ is the first of the last $u_{\alpha}(\lambda_i )$ cars having preference at least $\lambda_i$), since otherwise there would exist an element smaller than $i$ in $\tilde{T}_{\lambda_i} \subseteq \tilde{T}$, which would contradict the assumption $|\tilde{T}\cap [1,i-1]|=0$. As a consequence, among the first $i-1$ cars there are exactly $\sum_{j=\lambda_i}^{n}|\alpha |_j -u_{\alpha}(\lambda_i )=n-\lambda_i +1=|[\lambda_i ,n]|$ cars having preference at least $\lambda_i$. Since we are assuming that all such cars are able to park, if none of them were allowed to drive backwards, then all spots in $[\lambda_i ,n]$ must be filled by the first $i-1$ cars, against our current assumption. Thus, one of those cars must have parked by driving backwards, i.e., there exists $j\in [1,i-1]$ such that $a_j\geq \lambda_i$ and $r_j =j-1$, hence
	\[
	|\{ j\leq i\, |\, a_j \geq \lambda_i , r_j =j-1\} |\geq 1=|\tilde{T}\cap [1,i]|.
	\]
	
	{\color{blue}We are now left with the case $|\tilde{T}\cap [1,i-1]|\geq 1$}. For the last time in this proof, we consider two distinct cases.
	
	If $\lambda_{i-1}=\lambda_i$, since we are assuming that $[\lambda_i ,n]=[\lambda_{i-1} ,n]$ has not been entirely occupied by the first $i-1$ cars, our induction hypothesis (statements (a) and (b) combined) implies that
	\[
	|\{ j\leq i-1\, |\, a_j \geq \lambda_{i-1} , r_j =j-1\} |>|\tilde{T}\cap [1,i-1]|,
	\]
	hence
	\begin{align*}
	|\{ j\leq i\, |\, a_j \geq \lambda_i , r_j =j-1\} |&\geq |\{ j\leq i-1\, |\, a_j \geq \lambda_{i-1} , r_j =j-1\} | \\
	&\geq 1+|\tilde{T}\cap [1,i-1]| \\
	&=|\tilde{T}\cap [1,i]|.
	\end{align*}
	
	{\color{blue}Now suppose that $\lambda_{i-1}>\lambda_i$.} Notice that $\lambda_{i-1}\neq \infty$, since $|\tilde{T}\cap [1,i-1]|\geq 1$. If there is some spot in $[\lambda_{i-1} ,n]$ which has not been occupied by the first $i-1$ cars, then we can use our induction hypothesis and we argue and conclude exactly as in the previous case. Otherwise, suppose that $[\lambda_{i-1} ,n]$ has been entirely filled by the first $i-1$ cars. This means that the spot $h$ which we are assuming to be unoccupied after the first $i$ cars have parked is such that $h\in [\lambda_i ,\lambda_{i-1}-1]$. Recall that we have proved above that $i\in \tilde{T}_{\lambda_i}$. Moreover, the fact that $\lambda_{i-1}>\lambda_i$ means that the smallest index $j\in U_{\alpha}$ such that $\tilde{T}_j$ contains elements $\leq i-1$ is \emph{strictly larger} than $\lambda_i$. Therefore, for all $s<i$, $s\notin \tilde{T}_{\lambda_i}$, that is, $i=\min \tilde{T}_{\lambda_i}$.	In other words, $c_i$ is the first of the last $u_{\alpha}(\lambda_i )$ cars having preference at least $\lambda_i$. Thus (repeating an argument similar to one of the above cases) among the first $i-1$ cars there are exactly $n-\lambda_i +1=|[\lambda_i ,n]|$ cars having preference at least $\lambda_i$. Since there is an empty spot in $[\lambda_i ,\lambda_{i-1}-1]$, and all previous cars were able to park, there is a car \emph{having preference in $[\lambda_i ,\lambda_{i-1}-1]$} which has parked by driving backwards, and so
	\begin{align*}
	|\{ j\leq i\, |\, a_j \geq \lambda_i , r_j =j-1\} |&\geq 1+|\{ j\leq i\, |\, a_j \geq \lambda_{i-1} , r_j =j-1\} | \\
	&\geq 1+|\{ j\leq i-1\, |\, a_j \geq \lambda_{i-1} , r_j =j-1\} | \\
	&\geq 1+|\tilde{T}\cap [1,i-1]| \\
	&=|\tilde{T}\cap [1,i]|,
	\end{align*}
	which finally concludes the inductive proof of statement (a) (i.e. inequality (\ref{key})) of our claim.
	
	Now it remains to complete the induction proof of statement (b), in the case when $i\in \tilde{T}$ (the case $i\notin \tilde{T}$ having already been shown). So suppose that $|\{ j\leq i\, |\, a_j \geq \lambda_i , r_j =j-1\} |=|\tilde{T}\cap [1,i]|$. By contradiction, assume that some spot in $[\lambda_i ,n]$ remains unoccupied after the first $i$ cars have parked, and denote with $h$ the smallest of those. Since $i\in \tilde{T}$, $c_i$ is among the last $u_{\alpha}(\lambda_i )$ cars having preference at least $\lambda_i$. As a consequence, there are at least $\sum_{j=\lambda_i}^{n}|\alpha |_j -u_{\alpha}(\lambda_i )=n-\lambda_i +1$ cars 
	having preference at least $\lambda_i$ among the first $i-1$ cars, and so (considering $c_i$ which has preference $a_i \geq \lambda_i$) there are at least $n-\lambda_i +2$ cars having preference in $[\lambda_i ,n]$ among the first $i$ cars. Since the available spots are only $n-\lambda_i +1$, this means that (among the first $i$ cars) there is at least one car with preference in $[\lambda_i ,n]$ which parks in $[1,\lambda_{i}-1]$ by driving backwards. Moreover, since $\lambda_i \leq h\leq n$ is an unoccupied spot, any such car must have preference in $[\lambda_i ,h-1]$. Denote with $c_m$ the last of such cars (so that $1\leq m\leq i$). Necessarily, we have $r_m =m-1$.
	
	Define a new parking rule $\rho '=r'_1 r'_2 \cdots r'_n$ by setting $r'_i =r_i$ for all $i\neq m$, and $r'_m =0$. Clearly, using the $\rho'$-Naples parking rules, the first $m-1$ cars park in the same spots they occupy using the $\rho$-Naples parking rule. 
	
	We now want to describe what happens when the cars having indices in $[m,i]$ and with preference smaller than $h$ try to park using each parking rule.
	
	Using the $\rho$-Naples parking rule (and also recalling that $\alpha \in PF_{n,\rho}$, and so all cars are able to park), the first of the above cars (i.e., $c_m$) parks in $[1,\lambda_i -1]$ and its preference satisfies $h>a_m\geq \lambda_i$; thus, all spots in $[\lambda_i ,a_m]$ have already been occupied by previous cars. Moreover, let $k_1$ be the number of the above cars having preference in $[a_m +1,h]$; by definition of $h$ and $m$
	, all such cars park in $[a_m +1,h-1]$. Finally, denote with $k_2$ the number of the above cars having preference in $[1,a_m ]$ and reaching spot $a_m$ when trying to park; clearly all such cars park in $[a_m +1,h-1]$ as well. As a consequence, we have that $k_1 +k_2 \leq |[a_m +1,h-1]|$.
	
	On the other hand, using the $\rho'$-Naples parking rule, and taking into consideration the same set of cars, we have that the first of such cars (again $c_m$) parks in $[a_m +1,h]$. Morevoer, there are $k_1$ cars among those having preference in $[a_m +1,h]$. Finally, there are at most $k_2$ cars having preference in $[1,a_m ]$ and reaching spot $a_m$ when trying to park (this is because the spot occupied by $c_m$ using the $\rho$-Naples parking rule might be occupied by one of the $k_2$ cars considered in the previous paragraph when using the $\rho'$-Naples parking rule).
	
	Therefore, using the $\rho'$-Naples parking rule, the number of cars with indices in $[m,i]$ and parking in $[a_m ,h]$ is at most $1+k_1 +k_2 \leq 1+|[a_m +1,h-1]|\leq |[a_m ,h]|$. Since, as we have already observed, all cars with indices less than $m$ park in the same spots using each of the two parking rules, we have that, using the $\rho'$-Naples parking rule, among the first $i$ cars, those with preference in $[1,h]$ are able to park in the first $h$ spots. Consequently, all cars having preference in $[h+1,n]$ (among the first $i$ cars) occupy the same spots as they did using the $\rho$-Naples parking rule. This allows us to conclude that, using the $\rho'$-Naples parking rule, the first $i$ cars are all able to successfully park. However, by definition of $\rho'$, we have that
	\[
	|\{ j\leq i\, |\, a_j \geq \lambda_i , r'_j =j-1\} |<|\{ j\leq i\, |\, a_j \geq \lambda_i , r_j =j-1\} |=|\tilde{T}\cap [1,i]|,
	\]
	and this contradicts (\ref{key}) for the parking rule $\rho'$. This terminates the proof of our claim.
	
	We are now finally ready to complete our proof. Applying our claim for $i=n$ gives
	\[
	|\textrm{supp}(\rho )|=|\{ j\leq n\, |\, r_j =j-1 \}|\geq |\{ j\leq n\, |\, a_j \geq \lambda_n , r_j =j-1 \}|\geq |\tilde{T}\cap [1,n]|=|\tilde{T}|,
	\]
	as desired.
\end{proof}

\begin{proposition}\label{attain}
	Given $\alpha \in PP_n$, let $\rho \in \Delta_n$ such that $r_i =i-1$ for all $i\in \tilde{T}$, and $r_i =0$ otherwise. Then $\alpha \in PF_{n,\rho}$.
\end{proposition}

\begin{proof}
	Given $i\leq n$, suppose that the first $i-1$ cars have all successfully parked. We now show that also car $c_i$ is able to successfully park.
	
	If the first $i-1$ cars have not occupied all spots in $[a_i ,n]$, then certainly car $c_i$ is able to park. Otherwise, suppose that, after the first $i-1$ cars have parked, $[a_i ,n]$ has been entirely filled, and let $h$ be the largest unoccupied spot. Clearly $h<a_i$. The $n-h$ spots in $[h+1,n]$ have all been occupied by cars having preference at least $h+1$ (as $h$ is still unoccupied). As a consequence, since $a_i >h$, the total number of cars having preference strictly larger than $h$ is at least $n-h+1$, which gives
	\[
	u_{\alpha}(h+1)=\sum_{h+1}^{n}|\alpha |_k -(n-(h+1)+1)\geq 1,
	\]
	i.e. $h+1 \in U_{\alpha}$. Moreover, since there have already been at least $n-h$ cars with preference strictly larger than $h$, car $c_i$ is one of the last $\sum_{h+1}^{n}|\alpha |_k -(n-h)=u_{\alpha}(h+1)$ cars having preference at least $h+1$, hence $i\in \tilde{T}_{h+1}\subseteq \tilde{T}$. From the hypothesis we thus have that $r_i =i-1$, and so car $c_i$ is certainly able to park.
\end{proof}

We now reconsider our running example once more to illustrate the result of the present section. Recall that $\alpha =(4,9,5,9,5,8,7,9,4,6)\in PP_{10}$, for which $U_{\alpha}=[2,9]$. The following table shows how to compute the sets $\tilde{T}_j$:

\begin{center}
	\begin{tabular}{cccc}
		\hline
		& $u_{\alpha}(j)$ & $\{i\, |\, a_i \geq j\}$ & $\tilde{T}_j$ \\
		\hline
		$j=9$ & 1 &  $\{ 2,4,8\}$ & $\{ 8\}$ \\
		$j=8$ & 1 & $\{ 2,4,6,8\}$ & $\{ 8\}$ \\
		$j=7$ & 1 & $\{ 2,4,6,7,8\}$ & $\{ 8\}$ \\
		$j=6$ & 1 & $\{ 2,4,6,7,8,10\}$ & $\{ 10\}$ \\
		$j=5$ & 2 & $\{ 2,3,4,5,6,7,8,10\}$ & $\{ 8,10\}$ \\
		$j=4$ & 3 & $[1,10]$ & $\{ 8,9,10\}$ \\
		$j=3$ & 2 & $[1,10]$ & $\{ 9,10\}$ \\
		$j=2$ & 1 & $[1,10]$ & $\{ 10\}$ \\
		\hline
	\end{tabular}
\end{center}

Therefore we obtain $\tilde{T}=\bigcup_{j\in U_{\alpha}}\tilde{T}_j =\{ 8,9,10 \}$, which means (by Theorem \ref{lower_bound}) that at least three cars need to park by driving backwards in order to have a successful strategy. Moreover, by Proposition \ref{attain}, we know that, taking $\rho =0000000789$, we have $\alpha \in PF_{10,\rho}$, and specifically we get the outcome map $\psi_{\rho}(\alpha )=(4,9,5,10,6,8,7,3,2,1)$. Notice that the total number of backward steps taken by all cars in this case is 13, and we already showed that there are strategies needing only 12 steps for this specific $\alpha$. This again confirms that the up-set $\mathcal{F}_{\alpha}$ contains some minimal elements that do not minimize $|\rho |$.

We also remark that in general the parking strategy described in Proposition \ref{attain} is not the only one minimizing the number of cars driving backwards. For instance, given $\beta=(3,3,3,2,3,3,3)\in PP_7$, we get that $\tilde{T}=\{ 7\}$ and $\rho=0000006$ is such that $\beta \in PF_{7,\rho}$. However, it is immediate to check that the parking strategies $0000400$ and $0000050$ are equally successful for $\beta$. Furthermore, rearranging the entries of a parking preference may change the number of strategies minimizing the number of cars driving backwards, as well as the size of the set $\tilde{T}$.

\section{The case $\rho \in \{ 0,1\}^n$}\label{zeroone}

In this final section we examine a particularly interesting special case, namely when the parking strategy permits at most one backward steps for each car. In other words, this means that each entry of $\rho$ is either 0 or 1. The next theorem gives a complete description of $\rho$-Naples parking functions in this case. The theorem is preceded by a simple lemma concerning 1-Naples parking functions.

\begin{lemma}\label{1-Naples}
	Let $\alpha \in PF_{n,1}$ and let $[p,q]\subseteq U_{\alpha}$ be a maximal interval. Then $|\alpha|_q \geq 2$ and $|\alpha |_j =1$ for all $p\leq j<q$. 
\end{lemma}

\begin{proof}
	The fact that $|\alpha |_q \geq 2$ is exactly statement (iv) of Lemma \ref{elementary_intervals}. Since $\alpha \in PF_{n,1}$, Proposition \ref{naples_excess} implies that $u_{\alpha}(j)=1$ for all $j\in U_{\alpha}$. Hence, for $p\leq j<q$, using (iii) of Lemma \ref{elementary_properties}, we have that $|\alpha |_j =u_{\alpha}(j)-u_{\alpha}(j+1)+1=1$.  
\end{proof}

\begin{theorem}\label{01}
	Let $\alpha =(a_1 ,a_2, \ldots ,a_n ) \in PP_n$ and $\rho =r_1 r_2 \cdots r_n  \in \{ 0,1\}^n$. We have that $\alpha \in PF_{n,\rho}$ if and only if, for all maximal intervals $[p,q]$ of $U_{\alpha}$, there exist indices $j_1 <j_2 <\cdots <j_m$, with $m\geq q-p+2$, such that $a_{j_1}=a_{j_2}$ and, for all $i=2,3,\ldots ,m$, $a_{j_i}=m-i+p$ and $r_{j_i}=1$. 
\end{theorem}

\begin{proof}
	We start by assuming that, for all maximal intervals of $U_{\alpha}$, the conditions in the statement of the theorem hold. So let $[p,q]$ a maximal interval of $U_{\alpha}$ and $j_1, j_2 ,\ldots ,j_m$ as above. After car $c_{j_1}$ attempts to park, spot $a_{j_1}$ is certainly occupied. Thus, when car $c_{j_2}$ tries to park, it finds spot $a_{j_2}=a_{j_1}$ occupied; since $r_{j_2}=1$, it then tries to park into spot $a_{j_2}-1=a_{j_3}$, occupying it if necessary. The same argument can be repeated for all cars $c_{j_h}$, for $h\in [2,m]$, thus in particular spot $a_{j_m}-1=p-1$ is certainly occupied. Since this holds for every maximal interval of $U_{\alpha}$, we can invoke Lemma \ref{p-1} to conclude that $\alpha \in PF_{n,\rho}$.
	
	Vice versa, suppose now that $\alpha \in PF_{n,\rho}$, and let $[p,q]$ be a maximal interval of $U_{\alpha}$. We know from Lemma \ref{elementary_intervals} that $|\alpha |_{p-1}=0$ and $u_{\alpha}(p-1)=0$. 
	Since $\rho \in \{ 0,1\}^n$, cars with preference in $[p,n]$ can only park in $[p-1,n]$. Moreover, $\sum_{i=1}^{p-2}|\alpha |_i =p-2-u_{\alpha}(p-1)=p-2$, hence the cars having preference in $[1,p-2]$ occupy precisely the spots from 1 to $p-2$, and so the spots from $p-1$ to $n$ are entirely occupied by the cars having preference in $[p,n]$.
	
	From Proposition \ref{upset} we have that $\alpha \in PF_{n,1}$, and so Lemma \ref{1-Naples} tells us that $|\alpha |_q \geq 2$ and $|\alpha |_j \geq 1$, for all $j\in [p,q-1]$.
	
	Now consider spot $p-1$: it is occupied by a car having preference at least $p$, and $\rho \in \{ 0,1\}^n$, so there exists a car $c_{h_1}$ such that $a_{h_1}=p$, $r_{h_1}=1$ and $\psi_{\rho}^{(\alpha )}(h_1 )=p-1$. Moreover, when $c_{h_1}$ parks, spot $p$ must have already been occupied by a car $c_{h_2}$ (with $h_2 <h_1$) such that $a_{h_2}=p+1$, $r_{h_2}=1$ and $\psi_{\rho}^{(\alpha )}(h_2 )=p$. 
	This same argument can be iterated, thus generating a decreasing chain of natural numbers $h_1 >h_2 >h_3> \cdots$ whose generic element $h_{\mu +1}$ is such that $a_{h_{\mu +1}}=p+\mu $, $r_{h_{\mu +1}}=1$ and $\psi_{\rho}^{(\alpha )}(h_{\mu +1})=p+\mu -1$. Since the above decreasing chain must be finite, there exists $m$ such that $a_{h_m}=p+m-2=a_{h_{m-1}}$. The substitution $h_i =j_{m+1-i}$ gives the thesis. We conclude by observing that, since $|\alpha |_j \geq 1$, for all $j\in [p,q-1]$, then necessarily $a_{j_1}=a_{j_2}=p+m-2\geq q$, hence $m\geq q-p+2$.   
\end{proof}

As usual, we propose an example to illustrate our result. Consider the parking preference $\alpha =(5,10,11,1,5,11,10,4,3,9,10,8,3)\in PP_{13}$, with $U_{\alpha}=\{ 3\} \cup [8,10]$. It is easy to see that $\alpha$ is in fact a 1-Naples parking function. Now let $\rho =0000011001011\in \Delta_{13}$. For the maximal interval $\{ 3\}$ of $U_{\alpha}$ we have $a_9 =a_{13}=3$, with $r_{13}=1$, and for the maximal interval $[8,10]$ we have $a_3 =a_6 =11$, $a_7 =10$, $a_{10}=9$, $a_{12}=8$, with $r_6 =r_7 =r_{10}=r_{12}=1$. Therefore, Theorem \ref{01} assures that $\alpha \in PF_{13,\rho}$. Notice that, for the maximal interval $[8,10]$, we could also choose the sequence of indices $2,7,10,12$ (which means that the sequence of indices mentioned in Theorem \ref{01} is in general not unique).

\section{Further work}\label{further}
  
The notion of $\rho$-parking function that we have introduced in the present paper is still far from being completely understood. As a matter of fact, we have just scratched the surface of the subject, and a lot of intriguing open questions remain to be answered. Here we just mention a few possible avenues for future research.

First of all, we have not been able to fully understand the minimal elements of realizers in the lattices of parking strategies. This is the most natural open problem that directly arises from our investigations.

In Section \ref{zeroone} we have completely characterized $\rho$-Naples parking functions in the specific case when the entries of $\rho$ are all equal to 0 or 1. It may be interesting to find other parking strategies for which a full description of successful preferences can be provided. Natural candidates are parking strategies having a small fixed number (like one or two) of nonzero entries, or the parking strategies suggested in Chapter 1.13 of \cite{CCHJR} (which, suitable translated into our notation, for a given length $n$, are the parking strategy defined by setting $r_i =i-1$ when $i\leq \lfloor \frac{n+2}{2}\rfloor$ and $r_i =n-i+1$ when $i>\lfloor \frac{n+2}{2}\rfloor$).

By definition, the set of parking strategies coincides with the set of \emph{inversion sequences}, which in turn are a well known encoding of permutations. In particular, our lattices of parking strategies are isomorphic (actually, equal) to the middle orders on permutations introduced in \cite{BFT}. Is it possible to relate the permutation associated with a given parking strategy $\rho$ with properties of the set of $\rho$-Naples parking functions? 

As we briefly mentioned in the Introduction, classical parking functions have appeared in various mathematical contexts. It would be interesting to investigate whether the many generalizations of parking functions that have been defined in recent years admit similar treatments. In fact, most of such generalizations have been motivated by a modification of the parking rule (or more generally of the parking model), so it is not immediately apparent if they can have any interpretation or application in other contexts. Pursuing this study in the case of our $\rho$-Naples parking function seems to us an intriguing and stimulating project.

\end{document}